\documentclass[10pt]{amsart}

\usepackage[colorlinks]{hyperref}
\usepackage{color,graphicx,shortvrb}
\usepackage[latin 1]{inputenc}

\usepackage[active]{srcltx} 

\usepackage{enumerate}
\usepackage{amssymb}

\newtheorem{theorem}{Theorem}[section]

\newtheorem{lemma}[theorem]{Lemma}

\newtheorem{proposition}[theorem]{Proposition}

\def\11{\textbf{$1$}}

\DeclareGraphicsExtensions{.jpg,.pdf,.png,.eps}

\begin{document}

\title[$\ell_\infty(\Gamma)$ satisfies the Mazur-Ulam property]{Extending surjective isometries defined on the unit sphere of $\ell_\infty(\Gamma)$}

\author[A.M. Peralta]{Antonio M. Peralta}

\address{Departamento de An{\'a}lisis Matem{\'a}tico, Facultad de
Ciencias, Universidad de Granada, 18071 Granada, Spain.}
\email{aperalta@ugr.es}




\date{}

\begin{abstract} Let $\Gamma$ be an infinite set equipped with the discrete topology. We prove that the space $\ell_{\infty}(\Gamma),$ of all complex-valued bounded functions on $\Gamma$, satisfies the Mazur-Ulam property, that is, every surjective isometry from the unit sphere of $\ell_{\infty}(\Gamma)$ onto the unit sphere of an arbitrary complex Banach space $X$ admits a unique extension to a surjective real linear isometry from $\ell_{\infty}(\Gamma)$ to $X$.
\end{abstract}

\maketitle
\thispagestyle{empty}

\section{Introduction}

A result established by D. Tingley in 1987 proves that a surjective isometry $\Delta$ between the unit spheres of two finite dimensional Banach spaces preserves antipodal points, that is, $f(-x)=-f(x)$ for every $x$ in the unit sphere of the domain space (see \cite{Ting1987}). This contribution has served as stimulus and motivation to the growing interest on the so-called Tingley's problem which can be stated as follows: Let $X$ and $Y$ be Banach spaces whose unit spheres are denoted by $S(X)$ and $S(Y)$, respectively. Suppose $\Delta : S(X) \to S(Y)$ is a surjective isometry. Does $\Delta$ admits an extension to a surjective real linear isometry from $X$ onto $Y$?\smallskip

Tingley's problem remains open even in the case in which $X$ and $Y$ are $2$-dimensional Banach spaces. A considerable number of interesting results have shown that Tingley's problem admits a positive answer in ``classical'' Banach spaces like $\ell^p (\Gamma)$ spaces with $1\leq p\leq \infty$ (G.G. Ding \cite{Ding2002,Di:p,Di:8} and \cite{Di:1}), $L^{p}(\Omega, \Sigma, \mu)$ spaces, where $(\Omega, \Sigma, \mu)$ is a $\sigma$-finite measure space and $1\leq p\leq \infty$ (D. Tan \cite{Ta:8, Ta:1} and \cite{Ta:p}), $C_0(L)$ spaces (R.S. Wang \cite{Wang}), finite dimensional polyhedral Banach spaces (V. Kadets and M. Mart{\'i}n \cite{KadMar2012}), finite dimensional C$^*$-algebras and finite von Neumann algebras (R. Tanaka \cite{Tan2017b}), compact linear operators on a complex Hilbert space (A.M. Peralta and R. Tanaka \cite{PeTan16}), trace class operators (F.J. Fern{\'a}ndez-Polo, J.J. Garc{\'e}s, A.M. Peralta and I. Villanueva \cite{FerGarPeVill17}), bounded linear operators on a complex Hilbert space, weakly compact JB$^*$-triples and atomic JBW$^*$-triples (F.J. Fern{\'a}ndez-Polo, A.M. Peralta \cite{FerPe17, FerPe17b, FerPe17c}), and more recently von Neumann algebras (F.J. Fern{\'a}ndez-Polo, A.M. Peralta \cite{FerPe17d}), among others.\smallskip

A Banach space $X$ satisfies the Mazur-Ulam property if for every Banach space $Y$, Tingley's problem admits a positive solution for every surjective isometry $\Delta : S(X) \to S(Y)$.\smallskip

Let $\Gamma$ be an infinite set (equipped with the discrete topology). Accordingly to the standard notation, $c_0(\Gamma)$ will denote the space of all functions $x:\Gamma \to \mathbb{C}$ such that, for all $\varepsilon >0$, the set $\{ n\in \Gamma : |x(n)|\geq \varepsilon \}$ is finite. We consider $c_0(\Gamma)$ as a Banach space equipped with the supremum norm. In a recent contribution, we prove that the space $c_0(\Gamma)$ satisfies the \emph{Mazur-Ulam property}, that is, every surjective isometry from the unit sphere of $c_0(\Gamma)$ to the unit sphere of an arbitrary complex Banach space $X$ admits a unique extension to a surjective real linear isometry from $c_0(\Gamma)$ onto $X$ (see \cite{JVMorPeRa2017}).\smallskip

The main result in \cite{JVMorPeRa2017} provides what is probably the first example of a complex Banach space satisfying the Mazur-Ulam property. Prior contributions showed that the spaces $c_0(\Gamma,\mathbb{R}),$ of real null sequences, and $\ell_{\infty}(\Gamma,\mathbb{R}),$ of all bounded real-valued functions on a discrete set $\Gamma$, satisfy the Mazur-Ulam property (see \cite[Corollary 2]{Ding07}, \cite[Main Theorem]{Liu2007}). Another examples of Banach spaces satisfying the Mazur-Ulam property are the space $C(K,\mathbb{R}),$ of all real-valued continuous functions on a compact Hausdorff space $K$ \cite[Corollary 6]{Liu2007}, and the spaces $L^{p}((\Omega, \Sigma, \mu), \mathbb{R})$ of real-valued measurable functions on an arbitrary $\sigma$-finite measure space $(\Omega, \Sigma, \mu)$ for all $1\leq p\leq \infty$ \cite{Ta:1,Ta:8,Ta:p}. Although, a surjective linear isometry between the unit spheres of two complex Banach spaces need not admit an extension to a complex linear surjective isometry between the spaces (consider, for example, the conjugation on $S(\mathbb{C})$), the recent contributions on Tingley's problem for (complex) sequence spaces and operator algebras (compare \cite{Tan2016, Tan2017, Tan2017b, FerGarPeVill17, FerPe17b, FerPe17d} and the recent reference \cite{JVMorPeRa2017}) show the interest and attractiveness of the study of the Mazur-Ulam property in the setting of complex Banach spaces. \smallskip

It is conjectured in \cite{JVMorPeRa2017} that the space $\ell_{\infty}(\Gamma)$ also satisfies the Mazur-Ulam property. However, the techniques in the just quoted paper are not enough to prove this affirmation. In this note we provide a proof for this conjecture, and we confirm that the space $\ell_{\infty}(\Gamma)$ satisfies the Mazur-Ulam property. The main result reads as follows:

\begin{theorem}\label{t ellinfty satifies the MU} Let $\Gamma$ be an infinite set. Then, the space $\ell_{\infty}(\Gamma)$ satisfies the Mazur-Ulam property, that is, for each Banach space $X$, every surjective isometry $\Delta: S(\ell_{\infty}(\Gamma)) \to S(X)$ admits a unique extension to a surjective real linear isometry from $\ell_{\infty}(\Gamma)$ onto $X$.
\end{theorem}

The strategy in this note is to improve the geometric study developed in \cite{JVMorPeRa2017} to establish that $c_0(\Gamma)$ satisfies the Mazur-Ulam property. Here, instead of considering minimal projections in $c_0(\Gamma)$, we consider general non-zero partial isometries in $\ell_{\infty}(\Gamma)$. We shall show that a surjective isometry $\Delta: S(\ell_{\infty}(\Gamma)) \to S(X)$ maps every finite family of mutually orthogonal non-zero partial isometries in $\ell_{\infty}(\Gamma)$ to a completely $M$-orthogonal set in $X$ (Proposition \ref{new 3.7 with mo partial isometries}). This geometric argument is combined with the fact that we can define a real linear product $\odot$ on $\ell_{\infty}(\Gamma)$ satisfying $\lambda \Delta (v) = \Delta (\lambda\odot v)$, for every non-zero partial isometry $v\in \ell_{\infty}(\Gamma)$ and every $\lambda \in \mathbb{C}$ with $|\lambda|=1$ (see Lemma \ref{l takes out scalar with the new multiplication on projections}), to show a real linear behavior of the homogeneous extension of $\Delta$ on algebraic elements. The norm density of algebraic elements provides the final argument. The geometric arguments in this note are completely new compared with those in \cite{JVMorPeRa2017}. In striving for conciseness, we shall base our starting point on some of the ideas and arguments developed in \cite{JVMorPeRa2017,Liu2007} and \cite{Di:8}.

\section{Proof of the main theorem}

We begin gathering some results established in \cite{JVMorPeRa2017}. Henceforth, for $k$ in $\Gamma$, the symbol $e_n$ will denote the function in $\ell_{\infty}(\Gamma)$ satisfying $e_n(n) =1$ and $e_n(k) = 0,$ for all $k\in \Gamma$ with $n\neq k$. Given $n\in \Gamma$ and $\lambda\in \mathbb{T},$ we set $$A(n,\lambda):=\{ x\in S(\ell_{\infty} (\Gamma)) : x(n) = \lambda\}.$$ It is known that $A(n,\lambda)$ is a maximal norm closed face of the closed unit ball of $\ell_{\infty}$.\smallskip

Henceforth, the closed unit ball of a Banach space $X$ will be denoted by $\mathcal{B}_{X}$.\smallskip

The next lemma was established in \cite{JVMorPeRa2017}.

\begin{lemma}\label{l support}\cite[Lemma 2.1 and Proposition 3.4]{JVMorPeRa2017} Let $\Gamma$ be an infinite set, let $X$ be a Banach space, and let $\Delta : S(\ell_{\infty} (\Gamma))\to S(X)$ be a surjective isometry. Then, for each $n\in \Gamma$ and each $\lambda\in \mathbb{T},$ the set $$\hbox{\rm supp}(n,\lambda) := \{\varphi\in X^* : \|\varphi\|=1,\hbox{ and } \varphi^{-1} (\{1\})\cap S(X) = \Delta(A(n,\lambda)) \}$$ is a non-empty weak$^*$-closed face of $\mathcal{B}_{\ell_{\infty} (\Gamma)}$.\smallskip

Furthermore, if $n_0$ is an element in $\Gamma$ and $\varphi$ is an element in $\hbox{\rm supp}(n_0,\lambda)$ with $\lambda\in \mathbb{T}$, then $\varphi\Delta (x) =0$ for every $x\in S(\ell_{\infty} (\Gamma))$ with $x(n_0)=0$.$\hfill\Box$
\end{lemma}

The next lemma is a straight consequence of Lemma \ref{l support}. We recall that an element $v$ in $\ell_{\infty}(\Gamma)$ is a partial isometry if and only if $v v^*$ is a projection (i.e. $|v(k)| (1-|v(k)|)=0,$ for all $k\in \Gamma$).

\begin{lemma}\label{l supp(n, lambda) separates points} Let $\Gamma$ be an infinite set, let $X$ be a complex Banach space, and let $\Delta: S(\ell_{\infty}(\Gamma)) \to S(X)$ be a surjective isometry. Let $v$ be a non-zero partial isometry in $\ell_{\infty}(\Gamma)$, and let $n$ be an element in $\Gamma$. Suppose that $\varphi \Delta(v) =0$ for every $\varphi\in \hbox{supp}(n,\lambda),$ and every $\lambda\in \mathbb{T}$. Then $v(n)=0$.$\hfill\Box$.
\end{lemma}

Let $k$ be an entire number bigger than or equal to 2. Following the notation in \cite{JVMorPeRa2017}, we shall say that a set $\{x_1,\ldots,x_k\}$ in a Banach space $X$ is \emph{completely $M$-orthogonal} if $$\Big\|\sum_{j=1}^k \alpha_j x_j \Big\| = \max\{ \|\alpha_j x_j\| : 1\leq j\leq k\},$$ for every $\alpha_1,\ldots,\alpha_k$ in $\mathbb{C}$. It is known that a subset $\{x_1,\ldots,x_k\}$ in the unit sphere of a complex normed space $X$ is completely $M$-orthogonal if and only if the equality $$\displaystyle \left\|\sum_{j=1}^k \lambda_j x_j \right\| = 1$$ holds for every $\lambda_1,\ldots,\lambda_k$ in $\mathbb{T}=\{\lambda\in \mathbb{C}: |\lambda| =1\}$ (see \cite[Lemma 3.4]{JVMorPeRa2017}). Actually the following subtle variant of the last statement is required for later purposes.

\begin{lemma}\label{l new charact of CM orthogo} Suppose $k\in \mathbb{N}$ with $k\geq 2$. Let $\{x_1,\ldots,x_k\}$ be a subset of the unit sphere of a complex normed space $X$. Then $\{x_1,\ldots,x_k\}$ is completely $M$-orthogonal if and only if the equality $$\displaystyle \left\|\sum_{j=1}^k \lambda_j x_j \right\| = 1$$ holds for $\lambda_1=1$ and every $\lambda_2,\ldots,\lambda_k$ in $\mathbb{T}=\{\lambda\in \mathbb{C}: |\lambda| =1\}$.
\end{lemma}

\begin{proof} The ``only if'' implication follows from \cite[Lemma 3.4]{JVMorPeRa2017}. For the ``if'' implication, let us take $\lambda_1,\ldots,\lambda_k$ in $\mathbb{T}$. By the assumptions we have $$ \Big\|\sum_{j=1}^k \lambda_j x_j \Big\| = | \lambda_1| \Big\|x_1+\sum_{j=2}^k \frac{\lambda_j}{\lambda_1} x_j \Big\| = 1.$$ Under this conditions, Lemma 3.4 in \cite{JVMorPeRa2017} gives the desired statement.
\end{proof}

The next result has been essentially borrowed from \cite{JVMorPeRa2017}.

\begin{proposition}\label{p gathered properties}\cite[Propositions 3.3, 3.6, and 3.7]{JVMorPeRa2017} Let $\Gamma$ be an infinite set, let $X$ be a complex Banach space, and let $\Delta: S(\ell_{\infty}(\Gamma)) \to S(X)$ be a surjective isometry. Then the following statements hold:\begin{enumerate}[$(a)$]\item For each $n\in \Gamma$ and each $\lambda\in \mathbb{T}$ we have $\Delta (\lambda e_n) \in \{\lambda \Delta (e_n), \overline{\lambda} \Delta (e_n)\}$;
\item If $\Delta (\lambda e_n) =\lambda \Delta (e_n)$ {\rm(}respectively, $\Delta (\lambda e_n) =\overline{\lambda} \Delta (e_n)${\rm)} for some $\lambda\in \mathbb{T}\backslash\mathbb{R}$, then $\Delta (\mu e_n) =\mu \Delta (e_n)$ {\rm(}respectively $\Delta (\mu e_n) =\overline{\mu} \Delta (e_n)${\rm)} for all $\mu\in \mathbb{T}$, in this case we define $\sigma_n : \mathbb{C}\to\mathbb{C},$ $\sigma_n (\alpha):=\alpha$ {\rm(}respectively, $\sigma_n (\alpha):=\overline{\alpha}${\rm)} for all $\alpha\in \mathbb{C}$;
\item Let $n_1,\ldots,n_k$ be different elements in $\Gamma$. Then the set $\{\Delta(e_{n_1}),\ldots,\Delta( e_{n_{k}})\}$ is completely $M$-orthogonal, and the identity $$\sum_{j=1}^{k} \sigma_{n_j}(\alpha_{n_j}) \Delta( e_{n_j} ) = \sum_{j=1}^{k}  \Delta(\alpha_j e_{n_j} ) = \Delta \left(\sum_{j=1}^k \alpha_j e_{n_j} \right)$$ holds for every $\alpha_1,\ldots,\alpha_k\in \mathbb{C}\backslash\{0\}$ with $\max\{|\alpha_1|,\ldots,|\alpha_k|\}=1$. $\hfill\Box$
\end{enumerate}
\end{proposition}

Let $F$ be a finite subset of $\Gamma$. Let us take a subset $\{\lambda_j : j\in F\}$ in $\mathbb{T}$. It follows from the above proposition that $$\left\| \sum_{j\in F} \lambda_j \Delta( e_{j} ) \right\| = \left\|\Delta \left(\sum_{j\in F} \sigma_j(\lambda_j) e_{j} \right)  \right\| = \left\|\sum_{j\in F} \sigma_j(\lambda_j) e_{j}  \right\| =1. $$ This can be applied to conclude that, in the case $\Gamma=\mathbb{N}$ the series $\displaystyle \sum_{n\geq 1} \Delta( e_{n} )$ is weakly unconditionally Cauchy. We recall that a series $\displaystyle\sum_{n\geq 1} x_n$ in a Banach space $X$ is called \emph{weakly unconditionally
Cauchy (w.u.C.)} if there exists $C>0$ such that for any finite subset
$F\subset{\mathbb{N}}$ and $\varepsilon_{n} \in  \mathbb{T}$ we have $\displaystyle \left\| \sum_{n\in F}
\varepsilon_{n} x_{n} \right\| \leq C,$ equivalently, for each $\varphi\in X^*$ the series $\displaystyle \sum_{n\geq 1} |\varphi(x_n)|$ converges (see \cite[Theorem 6 in page 44]{Die}). However, being w.u.C. is not enough to conclude that the series $\displaystyle \sum_{n\geq 1} \Delta( e_{n} )$ converges to an element of $X$ in an appropriate topology (we can consider, for example the canonical basis in $X=c_0$).\smallskip

Elements $a,b\in \ell_{\infty}(\Gamma)$ are said to be orthogonal (written $a\perp b$) if $a b =0$.\smallskip

Our next result widens our knowledge on the image of two orthogonal elements in $S(\ell_{\infty}(\Gamma))$ under a surjective isometry onto the unit sphere of another Banach space.

\begin{proposition}\label{p projection an finitely many minimal projections are cm orthog} In the hypothesis of Proposition \ref{p gathered properties}, let $v$ be a partial isometry in $\ell_{\infty}(\Gamma)$. Suppose that ${n_1},\ldots, {n_k}$ are different elements in $\Gamma$ such that $v$ is orthogonal to $e_{n_1},\ldots, e_{n_k}$. Then the set $\{\Delta(v), \Delta(e_{n_1}),\ldots, \Delta(e_{n_k}) \}$ is completely $M$-orthogonal. Furthermore, given $\alpha_1,\ldots,\alpha_k\in \mathbb{C}\backslash\{0\}$ with $\max\{|\alpha_1|,\ldots,|\alpha_k|\}= 1$, we have \begin{equation}\label{eq formula for a pi and finitely many minimal pi} \Delta (v)+\sum_{j=1}^{k} \sigma_{n_j} (\alpha_j) \Delta( e_{n_j} ) =\Delta (v)+\Delta\left(\sum_{j=1}^{k}  \alpha_j e_{n_j}\right) = \Delta \left(v+\sum_{j=1}^k \alpha_j e_{n_j} \right).
\end{equation}
\end{proposition}

\begin{proof} We observe that when $v=0$ our result follows from Proposition \ref{p gathered properties}. To prove the first statement let us take $\lambda_1,\ldots, \lambda_k$ in $\mathbb{T}$. By Proposition \ref{p gathered properties} we have $$\left\| \Delta (v)+\sum_{j=1}^{k} \sigma_{n_j} (\lambda_j) \Delta( e_{n_j} ) \right\| = \left\| \Delta (v)-\Delta\left(- \sum_{j=1}^{k} \lambda_j  e_{n_j} \right) \right\| = \left\| v + \sum_{j=1}^{k} \lambda_j  e_{n_j} \right\| = 1.$$ Lemma \ref{l new charact of CM orthogo} implies that the set $\{\Delta(v), \Delta(e_{n_1}),\ldots, \Delta(e_{n_k}) \}$ is completely $M$-ortho-gonal.\smallskip

We consider now the second statement. Let us pick $\alpha_1,\ldots,\alpha_k\in \mathbb{C}\backslash\{0\}$ with $\max\{|\alpha_1|,\ldots,|\alpha_k|\}$ $=$ $1$. Since  $\{\Delta(v), \Delta(e_{n_1}),\ldots, \Delta(e_{n_k}) \}$ is completely $M$-ortho-gonal, the element $\displaystyle \Delta (v)+\sum_{j=1}^{k} \sigma_{n_j} (\alpha_j) \Delta( e_{n_j} )$ lies in the unit sphere of $X$. Therefore there exists a function $x\in S(\ell_{\infty}(\Gamma))$ satisfying $\Delta(x) = \displaystyle \Delta (v)+\sum_{j=1}^{k} \sigma_{n_j} (\alpha_j) \Delta( e_{n_j} )$. We shall prove that $\displaystyle x = v+\sum_{j=1}^k \alpha_j e_{n_j}$.

We shall argue by induction on $k$. However, prior to the induction argument, we shall first establish some facts valid for an arbitrary $k$.
Let us begin with an element $e_m$ such that $e_m$ is orthogonal to $v$, $e_{n_1},\ldots, e_{n_k}$. Since, by the first statement, the set $\{\Delta(v), \Delta(e_{n_1}),\ldots, \Delta(e_{n_k}), \Delta(e_m) \}$ is completely $M$-orthogonal, it follows from Proposition \ref{p gathered properties} that $$|x(m)\pm 1| \!\leq\! \|x\pm e_m\| \!=\! \| \Delta (x) \pm \Delta(e_m)\|\! =\!\left\|\Delta (v)+\sum_{j=1}^{k} \sigma_{n_j} (\alpha_j) \Delta( e_{n_j} )\pm \Delta(e_m)\right\|=1.$$ Consequently, $|x(m)\pm 1| \leq 1$, and hence $x(m)=0$. That is $$ x(m) = 0, \hbox{ for all } m\in \Gamma \hbox{ such that } e_m \perp v, e_{n_1},\ldots, e_{n_k}.
$$

Take now $m\in \Gamma$ such that $|v(m)|=1$. Given $\varphi\in \hbox{supp}(m,v(m))$. Having in mind that $e_{n_j} (m)=0$, for all $1\leq j\leq k$, Lemma \ref{l support} implies that $$\varphi\Delta(x) = \varphi \Delta (v)+\sum_{j=1}^{k} \sigma_{n_j} (\alpha_j) \varphi \Delta( e_{n_j} ) = 1,$$ which proves that $\Delta(x)\in \varphi^{-1}\{1\}\cap S(X) = \Delta(A(m,v(m))),$ and thus $x(m) = v(m)$. Similar arguments are also valid for all $m\in\{n_1,\ldots, n_k\}$ with $|\alpha_m|=1$. Therefore,
$ x(m) = v(m), \forall m\in \Gamma \hbox{ with } |v(m)|=1,$ and $x(m) = \alpha_m,$ $\forall m\in\{n_1,\ldots, n_k\}$ with $|\alpha_m|=1.$\smallskip

We have proved that if $x$ is an element in $S(\ell_{\infty} (\Gamma))$ with $\Delta(x) = \displaystyle \Delta (v)+\sum_{j=1}^{k} \sigma_{n_j} (\alpha_j) \Delta( e_{n_j} )$ then \begin{equation}\label{eq zero in the orthogonal of all 2} x(m) = 0, \hbox{ for all } m\in \Gamma \hbox{ such that } e_m \perp v, e_{n_1},\ldots, e_{n_k},
\end{equation} and  \begin{equation}\label{eq one on the support 2} x(m) = v(m), \forall m\in \Gamma \hbox{ with } |v(m)|=1,
\end{equation} $$\hbox{ and  } x(m) = \alpha_m, \forall m\in\{n_1,\ldots, n_k\} \hbox{ with } |\alpha_m|=1.$$

We prove now the identity in \eqref{eq formula for a pi and finitely many minimal pi} by induction on $k$. Let us assume that $k=1$. If $|\alpha_1|=1$, the desired equality follows from \eqref{eq one on the support 2} above. We can therefore assume, via \eqref{eq zero in the orthogonal of all 2} and \eqref{eq one on the support 2}, that $0<|\alpha_1|<1$ and  $$\Delta(x) = \Delta (v) + \sigma_{n_1} (\alpha_1) \Delta(e_{n_1}),\hbox{ and } x= v+ x(n_1) e_{n_1}.$$ Having in mind that $\Delta (v)$ and $\Delta (e_{n_1})$ are completely $M$-orthogonal, we can find $y\in S(\ell_{\infty}(\Gamma))$ satisfying $$ \Delta(y) = \Delta (v) + \frac{\sigma_{n_1} (\alpha_1)}{|\alpha_1|} \Delta(e_{n_1}).$$ By applying \eqref{eq zero in the orthogonal of all 2} and \eqref{eq one on the support 2} we deduce that $y = v + \frac{\alpha_1}{|\alpha_1|} e_{n_1}.$ We also know that $$1-|\alpha_1| =\left|\sigma_{n_1} (\alpha_1)- \frac{\sigma_{n_1} (\alpha_1)}{|\alpha_1|} \right|=\|\Delta (x) - \Delta(y) \|=\|x-y\| = \left|x(n_1)  - \frac{\alpha_1}{|\alpha_1|} \right|.$$ By Proposition \ref{p gathered properties} and the fact that $\Delta (v)$ and $\Delta (e_{n_1})$ are completely $M$-orthogonal we get  $$1+|\alpha_1| =\left| \sigma_{n_1} (\alpha_1) + \frac{\sigma_{n_1} (\alpha_1)}{|\alpha_1|} \right| =\left\|\Delta (x) + \frac{\sigma_{n_1} (\alpha_1)}{|\alpha_1|} \Delta(e_{n_1}) \right\|=\left\|x+ \frac{\alpha_1}{|\alpha_1|} e_{n_1} \right\| $$ $$ = \max\left\{ 1, \left| x(n_1) + \frac{\alpha_1}{|\alpha_1|} \right| \right\}=\left| x(n_1) + \frac{\alpha_1}{|\alpha_1|} \right|.$$ The equalities $\left| x(n_1) + \frac{\alpha_1}{|\alpha_1|} \right| = 1+|\alpha_1|$ and $1-|\alpha_1| = \left|x(n_1)  - \frac{\alpha_1}{|\alpha_1|} \right|$ prove $x(n_1) =\alpha_1$, as desired. The concludes the proof of the case $k=1$.\smallskip

Suppose, by the induction hypothesis, that \eqref{eq formula for a pi and finitely many minimal pi} holds whenever ${n_1},\ldots, {n_m}$ are different elements in $\Gamma$ with $m\leq k$, and every partial isometry $w$ orthogonal to $e_{n_1},\ldots, e_{n_m}$. Let us assume that ${n_1},\ldots, {n_{k+1}}$ are different elements in $\Gamma$, and $v$ is a non-zero partial isometry orthogonal to $e_{n_1},\ldots, e_{n_{k+1}}$. Since the set $\{\Delta(v), \Delta(e_{n_1}),\ldots, \Delta(e_{n_{k+1}}) \}$ is completely $M$-orthogonal, there exists $x\in S(\ell_{\infty} (\Gamma))$ with $$\Delta(x) = \Delta (v)+\sum_{j=1}^{k+1} \sigma_{n_j} (\alpha_j) \Delta( e_{n_j} ),$$ where $\alpha_1,\ldots,\alpha_k\in \mathbb{C}\backslash\{0\}$ with $\max\{|\alpha_1|,\ldots,|\alpha_k|\}= 1$.\smallskip

By applying \eqref{eq zero in the orthogonal of all 2} and \eqref{eq one on the support 2} we deduce that \begin{equation}\label{eq expresion for x} x = v+\sum_{j=1}^{k+1} x(n_j) e_{n_j},
 \end{equation} and $x(n_j) = \alpha_j$ if $|\alpha_j|=1$. Therefore, if $|\alpha_{j_0}|=1$ for some $j_0\in \{1,\ldots,k+1\}$, applying the induction hypothesis with $k=2$ and $k$ we have $$\Delta (x) = \Delta (v)+ \sigma_{n_{j_0}} (\alpha_{j_0}) \Delta( e_{n_{j_0}} ) +\sum_{j=1, j\neq j_{0}}^{k+1} \sigma_{n_j} (\alpha_j) \Delta( e_{n_j} ) $$ $$= \Delta (v+ \alpha_{j_0}  e_{n_{j_0}} ) +\sum_{j=1, j\neq j_{0}}^{k+1} \sigma_{n_j} (\alpha_j) \Delta( e_{n_j} ) = \Delta\left( v+ \alpha_{j_0}  e_{n_{j_0}} + \sum_{j=1, j\neq j_{0}}^{k+1} \alpha_j  e_{n_j} \right),$$ because $v+ \alpha_{j_0}  e_{n_{j_0}}$ is a partial isometry orthogonal to $\{e_{n_j} : j\neq j_{0}, j_1\}$. This gives the desired statement.\smallskip

We can thus assume that $|\alpha_{j}|<1$, for all $1\leq j\leq k+1$. Pick an arbitrary index $j_0\in \{1,\ldots,k+1\}$. Arguing as above, we can find $y\in S(\ell_{\infty} (\Gamma))$ with $$\Delta(y) = \Delta (v) + \frac{\sigma_{n_{j_0}} (\alpha_{j_0})}{|\alpha_{j_0}|} \Delta(e_{n_{j_0}}) +\sum_{j=1, j\neq j_{0}}^{k+1} \sigma_{n_j} (\alpha_j) \Delta( e_{n_j} ).$$ Now, we mimic the arguments above to deduce that, since the set $\{\Delta(v), \Delta(e_{n_1}),$ $\ldots,$ $\Delta(e_{n_{k+1}}) \}$ is completely $M$-orthogonal, it follows from \eqref{eq zero in the orthogonal of all 2}, \eqref{eq one on the support 2} and the induction hypothesis that \begin{equation}\label{eq expression for y} y = v + \frac{\alpha_{j_0}}{|\alpha_{j_0}|} e_{n_{j_0}} +\sum_{j=1, j\neq j_{0}}^{k+1} \alpha_j  e_{n_j}.
\end{equation}

Applying \eqref{eq expresion for x} and \eqref{eq expression for y} and the fact that the set $\{\Delta(v), \Delta(e_{n_1}),\ldots, \Delta(e_{n_{k+1}}) \}$ is completely $M$-orthogonal, we get
$$1-|\alpha_{j_0}| =\left|\sigma_{n_{j_0}} (\alpha_{j_0})- \frac{\sigma_{n_{j_0}} (\alpha_{j_0})}{|\alpha_{j_0}|} \right|=\|\Delta (x) - \Delta(y) \|=\|x-y\| $$ $$= \max\{|x(n_j)-y(n_j)| : 1\leq j_0\leq k+1, j\neq j_0 \} \vee  \left| x(n_{j_0})  - \frac{\alpha_{j_0}}{|\alpha_{j_0}|} \right|\geq \left| x(n_{j_0})  - \frac{\alpha_{j_0}}{|\alpha_{j_0}|} \right|.$$ By Proposition \ref{p gathered properties} and the fact that the set $\{\Delta(v), \Delta(e_{n_1}),\ldots, \Delta(e_{n_{k+1}}) \}$ is completely $M$-orthogonal we also obtain:  $$1< 1+|\alpha_{j_0}| =\left| \sigma_{n_{j_0}} (\alpha_{j_0}) + \frac{\sigma_{n_{j_0}} (\alpha_{j_0})}{|\alpha_{j_0}|} \right| =\left\|\Delta (x) + \frac{\sigma_{n_{j_0}} (\alpha_{j_0})}{|\alpha_{j_0}|} \Delta(e_{n_{j_0}}) \right\| $$ $$=\left\|x+ \frac{\alpha_{j_0}}{|\alpha_{j_0}|} e_{n_{j_0}} \right\| = \max\left\{ 1, \left| x(n_{j_0}) + \frac{\alpha_{j_0}}{|\alpha_{j_0}|} \right| \right\}= \left| x(n_{j_0}) + \frac{\alpha_{j_0}}{|\alpha_{j_0}|} \right|.$$ By combining $\left| x(n_{j_0})  - \frac{\alpha_{j_0}}{|\alpha_{j_0}|} \right|\leq  1-|\alpha_{j_0}|$ and $\left| x(n_{j_0}) + \frac{\alpha_{j_0}}{|\alpha_{j_0}|} \right| = 1+|\alpha_{j_0}|$ it can be concluded that $\alpha_{j_0} = x(n_{j_0})$, which finished the induction argument and the proof.
\end{proof}

Given an  element $a\in \ell_{\infty}(\Gamma)$ and a scalar $\alpha$ in $\mathbb{C}$ we define $\alpha \odot a $ as the element in $\ell_{\infty}(\Gamma)$ whose $k$th component is $(\alpha \odot a )(k)= \sigma_k (\alpha) a(k)$. Clearly $\lambda \odot v $ is a partial isometry when $\lambda\in \mathbb{T}$ and $v$ is a partial isometry. Furthermore, $\alpha \odot a = \alpha a$ for every $a\in\ell_{\infty}(\Gamma)$ and $\alpha$ in $\mathbb{R}$.

\begin{lemma}\label{l takes out scalar with the new multiplication on projections} In the hypothesis of Proposition \ref{p gathered properties}, let $v$ be a non-zero partial isometry in $\ell_{\infty}(\Gamma)$. Then for each $\lambda\in \mathbb{T}$ we have $\lambda \Delta(v) = \Delta (\lambda \odot v)$.
\end{lemma}

\begin{proof} Since $\lambda \Delta(v)$ lies in the unit sphere of $X$, and $\Delta$ is surjective, there exists $x\in \ell_{\infty}(\Gamma)$ satisfying $\Delta (x) = \lambda \Delta(v)$. Applying that $v$ is a non-zero partial isometry we can find a non-empty subset $\Gamma_0\subseteq \Gamma$ such that $v= \displaystyle \sum_{j\in \Gamma_0} \xi_j e_j$, where $\xi_j\in \mathbb{T}$ for all $j$ and the latter series converges in the weak$^*$-topology of $\ell_{\infty}(\Gamma)$. We shall show that $x=\lambda \odot v = \displaystyle \sum_{j\in \Gamma_0} \sigma_j(\lambda) \xi_j e_j$.\smallskip

If $m\in \Gamma\backslash \Gamma_0$, by applying Proposition \ref{p gathered properties} we get $$1=\|v \pm e_m\|=\left\|\Delta(v)\pm \Delta(e_m) \right\|=\left\| \Delta(x) \pm \lambda \Delta(e_m) \right\| = \left\| \Delta(x) - \Delta(\mp \sigma_m(\lambda) e_m) \right\|$$ $$ =\left\| x \pm \sigma_m(\lambda) e_m \right\| = \max\{\sup_{j\in \Gamma_0} \{ |x(j)| \}, |x(m)\pm \sigma_m(\lambda)|\},$$ which implies that $|x(m)\pm \sigma_m(\lambda)|\leq 1$, and consequently $x(m)=0$.\smallskip

Now, take $m\in \Gamma_0$. By Propositions \ref{p projection an finitely many minimal projections are cm orthog} and \ref{p gathered properties} we know that $$\Delta(x) =\lambda \Delta(v-\xi_m e_m) + \lambda \Delta(\xi_m e_m)  =\lambda \Delta(v-\xi_m e_m) +  \Delta(\sigma_m (\lambda) \xi_m e_m).$$ Pick $\phi\in \hbox{supp} (m,\sigma_m(\lambda) \xi_m)$. Since $(v- \xi_m e_m)(m) =0$, it follows from the properties defining the support (compare Lemma \ref{l support}) that $$ \phi \Delta(x) = \lambda \phi \Delta(v) = \lambda  \phi \Delta(v-\xi_m e_m) +  \phi \Delta(\sigma_m(\lambda) \xi_m e_m) = 0 + 1,$$ which shows that $\phi \Delta (x) = 1$, and hence $$\Delta (x) \in  \phi^{-1} \{1\}\cap S(X) = \Delta(A(m,\sigma_m(\lambda) \xi_m),$$ witnessing that $x(m) = \sigma_m(\lambda) \xi_m,$ which concludes the proof.
\end{proof}

We can now prove that a surjective isometry $\Delta: S(\ell_{\infty}(\Gamma)) \to S(X)$ maps finite sets of mutually orthogonal non-zero partial isometries to completely $M$-orthogonal sets.

\begin{proposition}\label{new 3.7 with mo partial isometries} In the hypothesis of Proposition \ref{p gathered properties}, let $v_1,\ldots,v_k$ be mutually orthogonal non-zero partial isometries in $\ell_{\infty}(\Gamma)$. Then the set $\{\Delta(v_1),\ldots,\Delta(v_k) \}$ is completely $M$-orthogonal, and the identity \begin{equation}\label{eq identity finite combination of mo partial siometries in the sphere} \sum_{j=1}^{k} \alpha_j \Delta( v_{j} ) = \Delta \left( \sum_{j=1}^k \alpha_j \odot v_{j} \right),
\end{equation}
holds for every $\alpha_1,\ldots,\alpha_k\in \mathbb{C}\backslash\{0\}$ with $\max\{|\alpha_1|,\ldots,|\alpha_k|\}= 1$.
\end{proposition}

\begin{proof} We shall first show that the set $\{\Delta(v_1),\ldots,\Delta(v_k) \}$ is completely $M$-ortho-gonal. We argue by induction on $k$. We observe that the case $k=1$ is clear. We consider the case $k=2$. Let $\lambda_1,\lambda_2$ be elements in $\mathbb{T}$. By Lemma \ref{l takes out scalar with the new multiplication on projections} and the hypothesis on $\Delta$ we have $$\left\|\lambda_1 \Delta(v_1) + \lambda_2 \Delta(v_2) \right\| =  \left\|\Delta( \lambda_1\odot v_1)- \Delta(- \lambda_2\odot v_2)\right\| = \left\| \lambda_1\odot v_1 + \lambda_2\odot v_2\right\|=1.$$ Lemma 3.4 in \cite{JVMorPeRa2017} proves that $\{\Delta(v_1),\Delta(v_2) \}$ is completely $M$-orthogonal.\smallskip

We claim that, \begin{equation}\label{eq induction identity for 2} \lambda_1 \Delta(v_1) + \lambda_2 \Delta(v_2) = \Delta (\lambda_1 \odot v_1 + \lambda_2\odot v_2),
\end{equation} for all $\lambda_1,\lambda_2$ in $\mathbb{T}$. Namely, since $\{\Delta(v_1),\Delta(v_2) \}$ is completely $M$-orthogonal, the element $\lambda_1 \Delta(v_1) + \lambda_2 \Delta(v_2)$ belongs to $S(X)$, and thus we can find $x\in S(\ell_{\infty} (\Gamma)$ with $\Delta (x) = \lambda_1 \Delta(v_1) + \lambda_2 \Delta(v_2)$. Let us find two non-empty disjoint subsets $\Gamma_1,\Gamma_2\subseteq \Gamma$ such that $\displaystyle v_j =\sum_{k\in \Gamma_j} \xi_k e_k$ for $j=1,2$, where $\xi_k\in \mathbb{T}$ for all $k\in \Gamma_j$.\smallskip

Pick $m\in \Gamma\backslash(\Gamma_1\cup \Gamma_2)$. In this case, by applying Proposition \ref{p projection an finitely many minimal projections are cm orthog}, we get $$|x(m)\pm \lambda_1|\leq \|x\pm \lambda_1 e_m\| = \|\Delta(x) \pm \sigma_{m}(\lambda_1) \Delta(e_m) \| $$ $$= \|\lambda_1 \Delta(v_1) + \lambda_2 \Delta(v_2) \pm  \sigma_{m}(\lambda_1) \Delta(e_m)\| = \| \Delta (\lambda_1 \odot v_1 +\lambda_1 e_m) - \Delta (-\lambda_2 \odot v_2) \| $$ $$= \| \lambda_1 \odot v_1 +\lambda_1 e_m +\lambda_2 \odot v_2 \| = 1,$$ where in the antepenultimate equality we have applied Proposition \ref{p projection an finitely many minimal projections are cm orthog}. This shows that $|x(m)\pm \lambda_1|\leq 1$, and hence $x(m) =0$.\smallskip

Take now $m\in \Gamma_1$. Proposition \ref{p projection an finitely many minimal projections are cm orthog} also shows that $$\Delta (x) = \lambda_1 \Delta(v_1) + \lambda_2 \Delta(v_2)= \lambda_1 \Delta(v_1- \xi_m e_m) +\lambda_1 \Delta (\xi_m e_m) + \lambda_2 \Delta(v_2)$$ $$= \lambda_1 \Delta(v_1- \xi_m e_m) + \Delta (\sigma_m(\lambda_1) \xi_m e_m) + \lambda_2 \Delta(v_2).$$  Take $\varphi\in \hbox{supp} (m,\sigma_m(\lambda_1) \xi_m)$. Since $(v_1- \xi_m e_m) (m) = v_2(m) =0$, Lemma \ref{l support} assures that $$\varphi \Delta(x) = \lambda_1 \varphi\Delta(v_1- \xi_m e_m) + \varphi\Delta (\sigma_m(\lambda_1) \xi_m e_m) + \lambda_2 \varphi\Delta(v_2) =1,$$ which proves that $\Delta(x)\in \varphi^{-1} \{1\}\cap S(X) = \Delta(A(m,\sigma_m(\lambda_1) \xi_m))$, and thus $x(m) = \sigma_m(\lambda_1) \xi_m$. Similar arguments show that $x(m) = \sigma_m(\lambda_1) \xi_m$ for all $m\in \Gamma_2$. This finishes the proof of \eqref{eq induction identity for 2}.\smallskip

We resume now the induction argument. We may assume that $k\geq 2.$ Let us assume that $v_1,\ldots, v_k,v_{k+1}$ are mutually orthogonal non-zero partial isometries in $\ell_{\infty} (\Gamma)$, and suppose that $\lambda_1,\ldots, \lambda_{k+1}$ are elements in $\mathbb{T}$. By applying \eqref{eq induction identity for 2} and the induction hypothesis we deduce that $$\left\| \sum_{j=1}^{k+1} \lambda_j \Delta (e_j) \right\| = \left\| \Delta (\lambda_1 \odot v_1 + \lambda_2\odot v_2) + \sum_{j=3}^{k+1} \lambda_j \Delta (e_j) \right\| = 1,$$ because $\lambda_1 \odot v_1 + \lambda_2\odot v_2$ is a partial isometry in $\ell_{\infty}(\Gamma)$. Lemma 3.4 in \cite{JVMorPeRa2017} implies that the set $\{\Delta(v_1),\ldots,\Delta(v_{k+1}) \}$ is completely $M$-orthogonal, which finishes the induction argument.\smallskip

We shall next prove the equality in \eqref{eq identity finite combination of mo partial siometries in the sphere}. We argue by induction on $k$. In the case $k=1$, we know that $|\alpha_1|=1,$ and hence the desired statement follows from Lemma \ref{l takes out scalar with the new multiplication on projections}.\smallskip

Take $k\geq 2$, and let us assume, by the induction hypothesis, that \eqref{eq identity finite combination of mo partial siometries in the sphere} is true when the number of non-zero partial isometries is smaller than or equal to $k-1$. We take $\alpha_1,\ldots,\alpha_k$ in $\mathbb{C}\backslash\{0\}$ with $\max\{|\alpha_1|,\ldots,|\alpha_k|\}= 1.$ It follows from the first statement of this proposition that the set $\{\Delta(v_1),\ldots,\Delta(v_{k}) \}$ is completely $M$-orthogonal. Therefore the element $\displaystyle\sum_{j=1}^{k} \alpha_j \Delta(v_j)$ lies in the unit sphere of $X,$ and hence there exists $x$ in $S(\ell_{\infty} (\Gamma))$ satisfying $\Delta (x) = \displaystyle\sum_{j=1}^{k} \alpha_j \Delta(v_j)$.\smallskip

Suppose there exists two different indices $j_1\neq j_2$ in $\{1,\dots,k\}$ with $|\alpha_{j_1}|=|\alpha_{j_2}|=1$. Having in mind that $\alpha_{j_1} \odot v_{j_1} + \alpha_{j_2} \odot v_{j_2} $ is a partial isometry, we apply \eqref{eq induction identity for 2} and the induction hypothesis to deduce the following
$$ \Delta(x) = \alpha_{j_1} \Delta(v_{j_1}) + \alpha_{j_2} \Delta(v_{j_2})+\sum_{j=1, j\neq j_1,j_2}^{k} \alpha_j \Delta(v_j) $$ $$= \Delta(\alpha_{j_1} \odot v_{j_1} + \alpha_{j_2} \odot v_{j_2}) + \sum_{j=1, j\neq j_1,j_2}^{k} \alpha_j \Delta(v_j)$$ $$ = \Delta\left (\alpha_{j_1} \odot v_{j_1} + \alpha_{j_2} \odot v_{j_2} + \sum_{j=1, j\neq j_1,j_2}^{k} \alpha_j \odot v_j \right).$$ Therefore $$x = \alpha_{j_1} \odot v_{j_1} + \alpha_{j_2} \odot v_{j_2} + \sum_{j=1, j\neq j_1,j_2}^{k} \alpha_j \odot v_j.$$

We can therefore assume the existence of a unique $j_0$ in $\{1,\ldots,k\}$ such that $|\alpha_{j_0}|=1.$ We pick another $j_1 \in\{1,\ldots,k\}$ with $j_1\neq j_0$. Since the set $\{\Delta(v_1),$ $\ldots,$ $\Delta(v_{k}) \}$ is completely $M$-orthogonal, there exists $y\in S(\ell_{\infty}(\Gamma))$ satisfying $$\Delta(y) = \alpha_{j_0} \Delta(v_{j_0})+ \frac{\alpha_{j_1}}{|\alpha_{j_1}|} \Delta(v_{j_1}) + \sum_{j=1}^{k} \alpha_j \Delta(v_j).$$ Having in mind that $|\alpha_{j_0}|=1=\frac{\alpha_{j_1}}{|\alpha_{j_1}|}$, by the arguments in the previous paragraph we have $$ y = \alpha_{j_0} \odot v_{j_0} + \frac{\alpha_{j_1}}{|\alpha_{j_1}|} \odot v_{j_1} + \sum_{j=1, j\neq j_0,j_1}^{k} \alpha_j \odot v_j.$$

On the other hand, since $v_1,\ldots,v_k$ are mutually orthogonal non-zero partial isometries, there exist non-empty disjoint subsets $\Gamma_1,\ldots,\Gamma_k\subseteq \Gamma$ such that $\displaystyle v_j =\sum_{m\in \Gamma_j} \xi_m e_m$ for $j=1,\ldots,k$, where $\xi_m\in \mathbb{T}$ for all $m\in \Gamma_j$.\smallskip

Let us now pick $m\in \Gamma_{j_1}$. By applying that $\{\Delta(v_1),$ $\ldots,$ $\Delta(v_{k}) \}$ is a completely $M$-orthogonal set we get \begin{equation}\label{eq 24 one k+1 induction hypothesis} 1 - |\alpha_{j_1}|= \left| \alpha_{j_1} - \frac{\alpha_{j_1}}{|\alpha_{j_1}|}\right| =\|\Delta(x) - \Delta(y) \| =\|x-y\| \geq \left|x(m) - \frac{\sigma_m(\alpha_{j_1})}{|\alpha_{j_1}|} \xi_m \right|.
 \end{equation} By similar arguments, Proposition \ref{p projection an finitely many minimal projections are cm orthog}, Lemma \ref{l takes out scalar with the new multiplication on projections}, and the fact that the set $$\{\Delta(v_{j_1}- \xi_m e_m), \Delta(\xi_m e_m),\}\cup\{\Delta(v_{j_1}) : 1\leq j\leq k, {j_1}\neq j \}$$ is completely $M$-orthogonal, can be applied to get \begin{equation}\label{eq 24 two k+1 induction hypothesis} \left|x(m) + \frac{\sigma_m(\alpha_{j_1})}{|\alpha_{j_1}|} \xi_m \right| \leq \left\|x + \frac{\sigma_m(\alpha_{j_1})}{|\alpha_{j_1}|} \xi_m e_m \right\|=\left\|\Delta(x) + \frac{\alpha_{j_1}}{|\alpha_{j_1}|} \Delta(\xi_m e_m) \right\| \end{equation}
$$=\left\|\alpha_{j_1} \Delta(v_{j_1}- \xi_m e_m)+ \alpha_{j_1} \Delta(\xi_m e_m) +\sum_{j=1, j\neq j_1}^{k} \alpha_j \Delta(v_j)  + \frac{\alpha_{j_1}}{|\alpha_{j_1}|} \Delta(\xi_m e_m) \right\| $$  $$=\max\{ |\alpha_j| : j\neq j_1\}\vee \left| \alpha_{j_1} + \frac{\alpha_{j_1}}{|\alpha_{j_1}|}\right| = 1 + |\alpha_{j_1}|.$$ Combining \eqref{eq 24 one k+1 induction hypothesis} and \eqref{eq 24 two k+1 induction hypothesis} we establish that $x(m) = {\sigma_m(\alpha_{j_1})} \xi_m$. The arbitrariness of $j_1$ and $m\in \Gamma_1$ assures that $$x = \sum_{j=1}^k \sum_{m\in \Gamma_j} {\sigma_m(\alpha_{j_1})} \xi_m e_m = \sum_{j=1}^k \alpha_j \odot \left(\sum_{m\in \Gamma_j} \xi_m e_m\right) = \sum_{j=1}^k \alpha_j \odot v_{j},$$ which concludes the induction argument and the proof.
\end{proof}

We are now in position to prove the main result of this note.

\begin{proof}[Proof of Theorem \ref{t ellinfty satifies the MU}] Let $\Delta : S(\ell_\infty (\Gamma))\to S(X)$ be a surjective linear isometry. We consider the homogeneous extension $F:\ell_\infty (\Gamma)\to X$, defined by $F(0)=0$ and $F(x) = \|x\| \Delta( \frac{1}{\|x\|} x)$ for all $x\in X\backslash\{0\}.$\smallskip

Let us fix an arbitrary set $\{v_1,\ldots, v_k\}$ of mutually orthogonal non-zero partial isometries in $\ell_{\infty} (\Gamma)$. Suppose that $a,b\in \ell_{\infty} (\Gamma)\backslash\{0\}$ can be written in the form $\displaystyle a = \sum_{j=1}^k \alpha_j\odot v_j,$ $\displaystyle b = \sum_{j=1}^k \beta_j\odot v_j,$ where $\alpha_j,\beta_j\in \mathbb{C}\backslash \{0\}$. If $a-b = 0,$ then Proposition \ref{new 3.7 with mo partial isometries} assures that $F(a) = \frac{1}{\|a\|} \Delta(a) = \frac{1}{\|a\|} (- \Delta(-a)) = - \frac{1}{\|b\|} \Delta(b) = - F(b)$, and hence $F(a+b) = 0 = F(a) + F(b).$\smallskip

Let us assume that $a+b\neq 0$. By definition and Proposition \ref{new 3.7 with mo partial isometries} we have $$F(a) = \|a\| \Delta \left(\frac{1}{\|a\|} a\right)= \|a\| \Delta   \left(\sum_{j=1}^k \frac{\alpha_j}{\|a\|} \odot v_j \right) $$ $$= \|a\| \sum_{j=1}^k \frac{\alpha_j}{\|a\|} \Delta (v_j) = \sum_{j=1}^k \alpha_j \Delta (v_j),$$
$$F(b) = \|b\| \Delta \left(\frac{1}{\|b\|} b\right)= \|b\| \Delta   \left(\sum_{j=1}^k \frac{\beta_j}{\|b\|} \odot v_j \right) = \sum_{j=1}^k {\beta_j} \Delta (v_j),$$
$$F(a+b) = \|a+b\| \Delta \left(\frac{1}{\|a+b\|} (a+b)\right)= \|a+b\| \Delta   \left(\sum_{j=1}^k \frac{\alpha_j+\beta_j}{\|a+b\|} \odot v_j \right)=  \sum_{j=1}^k (\alpha_j+\beta_j) \Delta (v_j).$$ Therefore, $F(a) + F(b) = F(a) + F(b).$\smallskip

Let us observe that, given $a,b\in \ell_{\infty} (\Gamma)\backslash \{0\}$, we have $$ \|F(a)-F(b) \| = \left\| \|a\| \Delta \left(\frac{1}{\|a\|} a\right) - \|b\| \Delta \left(\frac{1}{\|b\|} b\right) \right\|$$ $$\leq \|a\| \left\|\Delta \left(\frac{1}{\|a\|} a\right) -  \Delta \left(\frac{1}{\|b\|} b\right) \right\| + \left\|  \Delta \left(\frac{1}{\|b\|} b\right) \right\| \|a-b\|  $$ $$ = \|a\| \left\|\frac{1}{\|a\|} a  -   \frac{1}{\|b\|} b \right\| +  \|a-b\| \leq 3 \|a-b\|,$$ and then $F$ is a Lipschitz mapping.\smallskip

Finally, given $a,b\in \ell_{\infty} (\Gamma)\backslash\{0\}$ and $\varepsilon >0$ we can find a set $\{v_1,\ldots,v_k\}$ of mutually orthogonal non-zero partial isometries and $\alpha_1,\beta_1,\ldots,\alpha_k,\beta_k \in \mathbb{C}\backslash\{0\}$ such that $\displaystyle \left\|a - a_k \right\|<\varepsilon$ and $ \left\|b - b_k \right\|<\varepsilon,$ where $\displaystyle a_k = \sum_{j=1}^k \alpha_j\odot v_j$, and $\displaystyle b_k=\sum_{j=1}^k \beta_j\odot v_j$. Since, by the arguments in the first part of this proof, we know that $F(a_k + b_k ) = F(a_k) + F(b_k),$ and $F$ is a Lipschitz mapping, we deduce, from the arbitrariness of $\varepsilon>0$, that $F(a+ b) = F(a) + F(b)$ for all $a,b\in \ell_{\infty} (\Gamma)$, which concludes the proof.
\end{proof}

\medskip\medskip

\textbf{Acknowledgements} Author partially supported by the Spanish Ministry of Economy and Competitiveness (MINECO) and European Regional Development Fund project no. MTM2014-58984-P and Junta de Andaluc\'{\i}a grant FQM375.

\end{document}